\documentclass[11pt, letterpaper]{amsart}

\usepackage{color}
\usepackage{amsfonts}
\usepackage{amssymb}
\usepackage{amsmath}
\newcommand{\diam}{{\rm diam}}

\newcommand{\pd}{\partial}

\newcommand{\re}{\mbox{\rm Re\,}}

\newcommand{\al}{\alpha}
\newcommand{\eps}{\varepsilon}

\newcommand{\vf}{\varphi}
\newcommand{\om}{\omega}
\newcommand{\Om}{\Omega}

\newcommand{\cB}{\mathcal B}

\newcommand{\cP}{\mathcal P}

\newcommand{\bC}{\mathbb C}
\newcommand{\bD}{\mathbb D}

\newcommand{\bR}{\mathbb R}

\newtheorem{theorem}{Theorem}[section]

\newtheorem{lemma}[theorem]{Lemma}
\newtheorem{prop}[theorem]{Proposition}
\theoremstyle{definition}
\newtheorem{remark}[theorem]{Remark}

\newtheorem{ex}[theorem]{Example}

\let\comma=,
{\catcode`\,=\active \gdef,{\relax\ifmmode\comma\else{\upshape\comma}\fi}

\numberwithin{equation}{section}

\newcounter{vremennyj}

\begin{document}
\title[The boundedness of integration operator on $H^\infty(\Omega)$ ]{Uniform approximation of Bloch
functions and the boundedness of the integration operator on $H^\infty$}%
{%$H^1$
%Entropy bumps  and applications, the non-homogeneous setting
%Entropy ``bumping'' of the two weight Muckenhoupt $A_2$ conditio
%Calder\'on--Zygmund
%operators
}
\date{\today}
\author[W.~Smith]{Wayne Smith}
\thanks{}
\address{Department of Mathematics, University of Hawaii, USA}
\email{wayne@math.hawaii.edu \textrm{(W. \ Smith)}}
\author[D.~Stolyarov]{Dmitriy M. Stolyarov}
\thanks{DS is partially supported by the RSF grant 14-41-00010; DS is grateful to Alexander Logunov for the discussion}
\address{Department of Mathematics, Michigan State University, East Lansing, MI. 48823, and Chebyshev Laboratory, St. Petersburg, Russia, and St. Petersburg Department of Steklov Mathematical Institute, St. Petersburg, Russia}
\email{dms@math.msu.edu \textrm{(D. \ Stolyarov)}}
\author[A.~Volberg]{Alexander Volberg}
\thanks{AV  is partially supported  by the Oberwolfach Institute for Mathematics, Germany, and  by the NSF grants DMS-1265549 and DMS-1600065}
\address{Department of Mathematics, Michigan Sate University, East Lansing, MI. 48823}
\email{volberg@math.msu.edu \textrm{(A.\ Volberg)}}
%\maketitle
\makeatletter
\@namedef{subjclassname@2010}{
  \textup{2010} Mathematics Subject Classification}
\makeatother
% 42B	Harmonic analysis in several variables
% 42B20	Singular and oscillatory integrals (Calder?on-Zygmund, etc.)
% 42B35	Function spaces arising in harmonic analysis
% 47A	General theory of linear operators
% 47A30	Norms (inequalities, more than one norm, etc.)
%{30E20, 47B37, 47B40, 30D55.}
%
% 30D55	$H^p$-classes (1980-2009)
% 30E20	Integration, integrals of Cauchy type, integral representations of analytic functions
%
% 47B   	Special classes of linear operators
% 47B37	Operators on special spaces (weighted shifts, operators on sequence spaces, etc.)
% 47B40	Spectral operators, decomposable operators, well-bounded operators, etc.
\begin{abstract}
We obtain a necessary and sufficient condition for the operator of integration to be bounded on $H^\infty$ in a simply connected domain. The main ingredient of the proof is a new result on uniform approximation of Bloch functions.
\end{abstract}
\maketitle

\section{Introduction and statements of results}
\label{approx1}

Consider any simply connected domain $O$ in the complex plane. Fix $p\in O$ and consider the operator
of complex integration defined on $H(O)$, the set of functions analytic in~$O$:
$$
J_p f\,(z)= \int_p^z f(\zeta)\, d \zeta, \quad f\in H(O), z\in O.
$$

This operator is related to a generalized Volterra operator acting
on $H(\bD)$, where $\bD$ is the unit disk.  Let $g\in H(\bD)$ and
define the operator $T_g:H(\bD)\to H(\bD)$ by

$$
T_gf(w)=\int_0^w f(t)g'(t)\, dt, \quad f\in H(\bD),\,\,\zeta\in \bD.
$$
In the case that $g$ is univalent, the change of variable $\zeta=g(t)$ transforms the operator
$T_g$ on $H(\bD)$ to the operator $J_{g(0)}$ on $H(g(\bD))$.

The operator $T_g$ has been studied on many Banach spaces $X\subset H(\bD)$.  For such $X$, define
$$
T[X]=\{g\in H(\bD)\,:\, T_g \text{ is bounded on } X\}.
$$
C. Pommerenke's short proof of the analytic John-Nirenberg inequality
in \cite{Pom1}, based on his observation that $T[H^2]=$ BMOA, attracted considerable interest.
Subsequently, $T[X]$ has been identified for a variety of spaces $X$, including the Hardy spaces
($0<p<\infty$), Bergman spaces, and BMOA; see \cite{AC}, \cite{AS}, \cite{AS2} and \cite{SZ}.
The study of $T[H^{\infty}]$, where $H^\infty=H^\infty(\bD)$ is the usual space of bounded
analytic functions on $\bD$, was begun in \cite{AJS}.
We note also the paper \cite{NZ}, which gives sufficient conditions for the boundedness of the operators under investigation.
It is clear that
$T[H^{\infty}]\supseteqq$ BRV, the space of  functions analytic on $\bD$ with bounded radial variation
$$
\text{BRV}=\{g\in \text{H}(\bD)\,:\, \sup_\theta\int_0^1|g'(re^{i\theta})|\,dr<\infty\},
$$
and in the article \cite{AJS} it was conjectured that $T[H^{\infty}]=$ BRV.
In the case that $g$ in univalent, this becomes a conjecture about when the operator
$J_{g(0)}$ is bounded on $H^\infty(g(\bD))$.  Our main result, Theorem \ref{oper} below, confirms
this conjecture when the symbol $g$ is univalent.

Recently, a discussion between
Fedor Nazarov,  Paata Ivanisvili, Alexander Logunov, and one of the authors (D. Stolyarov), resulted in a
counterexample to the general conjecture in \cite{AJS} about when $T_g$ is bounded.  Thus it is now known
that BRV $\subsetneqq T[H^{\infty}]$.  We thank F. Nazarov,  P. Ivanisvili, and A. Logunov for permission
to include the counterexample at the end of this paper.

Recently a paper \cite{CPPR} addressing the action
 of the operator $T_g$  from a Banach space $X$ into $H^\infty$ has appeared in arXiv.
It has many interesting results, including
a characterization of when   $T_g$  is bounded on $H^\infty$:
\cite[Theorem 1.2]{CPPR}
$g\in T[H^\infty]$ if and only if $\sup_{z\in\bD}\|G_{g,z}\|_{\mathcal K}<\infty$, where
$$
\overline{G_{g,z}(w)}= \int_0^z g'(\zeta)\overline{K_\zeta(w)}\,d\zeta.
$$
Here $K_\zeta$ denotes the reproducing kernel for $H^2$ and $\mathcal K$ denotes the space of Cauchy transforms.

By the interior diameter of $O$  we understand the following quantity:
$$
\text{diam}_I \,O := \sup_{z_0,z_1\in O}\,\, \inf_{\gamma\in\Gamma(z_0,z_1)} \int_0^1 |\gamma'(t)| dt,
$$
where $\gamma\in\Gamma(z_0,z_1)$ means that $\gamma\colon [0,1]\to O$ is a smooth path with
$\gamma(0)=z_0$ and $\gamma(1)=z_1$.

\begin{theorem}
\label{oper}
Let $O$ be a simply connected domain in the plane, and let $p\in O$.
The operator $J_p$ is bounded on $H^\infty(O)$ if and only if  $\diam_I \,O<\infty$.
\end{theorem}

Our proof of Theorem \ref{oper} is based on a new result on uniform approximation of Bloch functions,
the connection being that when $g$ is univalent, $\log g'$ is a Bloch function.
Recall that a function $f\in H(\bD)$ is said to be a Bloch function if $(1-|z|)|f'(z)|$ is
bounded on $\bD$.
For the statements, we introduce notation for the usual partial differentiation operators
$$
\partial=\frac{\partial}{\partial z} =
 \frac12\left(\frac{\partial}{\partial x}-i\frac{\partial}{\partial y}\right), \quad
\bar{\partial}=\frac{\partial}{\partial \bar z} =
 \frac12\left(\frac{\partial}{\partial x}+i\frac{\partial}{\partial y}\right).
 $$
In particular, when $f$ is analytic $\partial f=f'$ and $\bar{\partial}f =0$.

Let $\Omega_{\al}^r$ denote the domain $\{z\colon |z|<r, \, \arg\, z\in (-\al/2, \al/2)\}$. We abbreviate $\Omega_\al =\Omega_\al^1$.
By $\cB(\Omega_\al^r)$ we denote the class of functions analytic in $\Om_\al^r$ and such that
\begin{equation}
\label{Bl}
|\pd F(z)| \le \frac{C_F}{|z|},\, z\in \Om_\al^r.
\end{equation}
For a harmonic function $u$ on $\Om_\beta$, denote by~$\tilde{u}$  the harmonic conjugate of~$u$
with $\tilde u (1/2) =0$.

\begin{theorem}
\label{realpart}
Let $0<\alpha<\beta<\pi$, $\eps>0$, and let $F\in \cB(\Om_\al^{1/2})$.  Then there exists a harmonic function $u$ in $\Om_\beta$ such that

\textup{1)} $|u(x)- \re F(x)|\le \eps,\, x\in (0, \delta(\eps)]$.

\textup{2)} $|\tilde u(z)| \le C(\eps,\al,\beta,C_F)<\infty$ for all $z\in \Om_{\beta}$.

\end{theorem}

In Section 2, we will assume  Theorem \ref{realpart} and use it
 to prove Theorem \ref{oper}.  The proof of Theorem \ref{realpart} will be given in Section 3.
The example showing that   BRV $\subsetneqq T[H^{\infty}]$
is in Section 4.

\medskip

{\it Notation for constants.} The letter $C$ will be used throughout the paper to denote various positive constants
which may vary at each occurrence but do not depend on the essential parameters.
The dependence of  $C$ on important variables will be often indicated by placing the
variables in parentheses.
For $X$ and $Y$ nonnegative quantities, the notation $X\lesssim Y$  or $Y\gtrsim X$
means $X\le CY$ for some inessential constant $C$. Similarly,  $X\approx Y$ means that
both $X\lesssim Y$ and $Y\lesssim X$ hold.

\noindent{\bf Acknowledgments.} We already mentioned that we are grateful to P. Ivanisvili, A. Logunov and F. Nazarov for the discussion (mentioned above) that led to the first example in Section \ref{examples}.
We are  thankful to the referee for making valuable remarks that significantly improved  the exposition, in particular, we are grateful to the referee for the remark at the end of Section 3.

\section{The proof of Theorem \ref{oper}, assuming Theorem
\ref{realpart}}

In this section we assume that Theorem \ref{realpart} holds and show that Theorem \ref{oper} is a consequence.

The proof of one implication in Theorem \ref{oper} does not require Theorem \ref{realpart}.
Suppose that
diam$_I\, O<\infty$, and let $z\in O$ and $f\in H^\infty(O)$ be arbitrary. Then, for
any smooth path $\gamma:[0,1]\to O$ connecting $p$ to $z$, we have
$$
\left|J_p f\,(z)\right|=\left|\int_\gamma f(\zeta)\,d\zeta\right|\le \|f\|_\infty\int_0^1|\gamma'(t)|\,dt.
$$
Thus, taking the infimum  over all such paths $\gamma$
shows that $J_p: H^\infty(O)\to H^\infty(O)$ is  a bounded operator with
$\|J_p\|\le \text{diam}_I \,O$.

For the other implication in Theorem \ref{oper},
assume that diam$_I\, O=\infty$.  In the case that $O=\mathbb{C}$, by considering its action on
the constant function 1, it is clear that $J_p$ is unbounded on $H^\infty(O)$.  On the other
hand, if $O$ is a 
proper subset of $\mathbb C$, let
 $\vf:\bD\to O$ be a Riemann
map, which we may assume is normalized so that $\vf'(0)=1$.
Since the interior diameter is infinite,
given an integer $N$ there exists a radius, which we may assume is $[0,1)$, such that
\begin{equation}
\label{geN}
\int_0^1|\vf'(x)| \,dx\ge N\,.
\end{equation}

Consider the function
\begin{equation}
\label{f}
f(z):= e^{-i(u(z)+i\tilde u(z))}\in \mathcal B(\Om_\beta),
\end{equation}
 where  $u, \tilde u$ satisfy Theorem \ref{realpart} with $F$ chosen as follows.
Fix some $\al <\beta < \pi$.
%Fix $\gamma= (\al+\beta)/2$.
We first denote by $\psi_\beta: \Om_\beta\to \bD$ the conformal map with $\psi_\beta(1/2) =0$ and $\psi_\beta(0) =1$.  Then $|z\psi_\beta'(z)|\approx|1-\psi_\beta(z)|$, for $z\in \Omega_\beta$  and $|z|\le 1/2$, where the
constants suppressed depend only on $\beta$.  Since also $|z\psi_\beta'(z)|\lesssim 1$, it follows that $|z\psi_\beta'(z)|\lesssim|1-\psi_\beta(z)|$, for $z\in \Omega_\beta$.
Hence restricting to $\Om_\al^{1/2}$ gives
$$
|z\psi_\beta'(z)|\lesssim1-|\psi_\beta(z)|^2, \quad z\in \Om_\al^{1/2},
$$
with constants depending only on $\al$ and $\beta$.
Now consider the composition
$$
F(z):= i \log \vf' \circ \psi_\beta (z)\,.
$$
Using the well known inequality  (see, for example, \cite[p. 9]{Pom}) that
$$
(1-|z|^2)\frac{|\vf''(z)|}{|\vf'(z)|}\le 6, \quad z\in \bD,
$$
for the  univalent function $\vf$,
it follows that the restriction to $\Om_\al^{1/2}$ of the function $F$ satisfies \eqref{Bl} (with $r=1/2$),
with constant $C_F=C(\al,\beta)$ depending only on $\al$ and $\beta$.
Thus Theorem~\ref{realpart} is applicable, and we get an approximate $\Phi= u+i\tilde u$  defined in $\Omega_\beta$.

By this theorem (with $\eps =\pi/4$)
$$
|\re F(x) - u(x)| \le \pi/4, \quad x\in (0,\delta(\pi/4)],
$$
and therefore we have that
$$
|\arg \vf'(t) - u(\psi_\beta^{-1}(t))|\le \pi/4, \quad r_\beta<t<1,
$$
where $r_\beta\in(0,1)$ and depends only on $\beta$.
From the Koebe distortion theorem and our assumption that $\vf'(0)=1$, there is a constant $C_1(\beta)$
such that
\begin{equation}
\label{lerbeta}
\int_0^{r_\beta}|\vf'(t)|\,dt\le C_1(\beta).
\end{equation}
We also have from  Theorem \ref{realpart} that
$$
|\tilde u|\le C_2(\al,\beta),\quad z\in \Om_\beta.
$$
 Therefore, if $f$ is the function from \eqref{f}, then function
$$
g:= f\circ \psi_\beta^{-1}\in H^\infty(\bD),
$$
with $\|g\|_\infty< \exp(C_2(\al,\beta))$.  We now estimate
\begin{align*}
 \re \int_{r_\beta}^1  g(x) \vf'(x) dx &= \re \int_{r_\beta}^1  e^{i (\arg \vf'(x) -u(\psi_\beta^{-1}(x))} |\vf'(x)| e^{\tilde u(\psi_\beta^{-1}(x))} dx\\
 &\ge \cos (\pi/4) e^{-C_2(\al,\beta)} \int_{r_\beta}^1 |\vf'(x)| dx \\
&\ge
 \cos (\pi/4) e^{-C_2(\al,\beta)} (N-C_1(\beta)),
\end{align*}
from \eqref{geN} and \eqref{lerbeta}.  Since also
$$
 \left| \re \int_0^{r_\beta}  g(x) \vf'(x) dx\right| \le C_1(\beta)\|g\|_\infty\le C_1(\beta)\exp(C_2(\al,\beta)),
 $$
it follows that
$$
\re \int_0^1 g(x) \vf'(x) dx\ge \cos (\pi/4) e^{-C_2(\al,\beta)} (N-C_1(\beta))-C_1(\beta)\exp(C_2(\al,\beta)).
$$
Since the integer $N$ was arbitrary, and $\|g\|_\infty< \exp(C_2(\al,\beta))$,
this means that the operator $J_p$ is unbounded on $H^\infty(O)$.
Theorem \ref{oper} is proved.

\section{The proof of Theorem \ref{realpart}}
\label{prooffixed}

We separate out the main part of the proof of
 Theorem \ref{realpart} into the following lemma.

\begin{lemma}
\label{fixed}
Let  $0<\al<\beta<\pi$, and let $\eps>0$. Given a function $F\in\cB( \Om_\al^{1/2})$, one can find analytic
$\Phi$ such that

\textup{1)} $|F(x)-\Phi(x)|\le \eps, \quad x\in (0, \delta(\eps)];$

\textup{2)} $\Phi\in\cB( \Omega_\beta)$ and $C_\Phi =C(\eps, \al,\beta, C_F)$.

%\textcolor{green}{I combined what was 2) and 3) into one statement 2), as in the next theorem.  Okay?}
\end{lemma}

\subsection{The proof of Theorem \ref{realpart}, assuming Lemma \ref{fixed}}
\label{prooffixed2}

Given $F=U+iV\in \cB(\Om_\al^{1/2})$,  consider its symmetrization $F^*(z)= (F(z) + \overline{F(\bar{z})})/2
= U^*(z)+iV^*(z)$. It obviously belongs to $\cB(\Om_\al^{1/2})$  as well,
and we apply Lemma \ref{fixed}
 to it to obtain a
function (let us call it) $\Phi^*$ which satisfies the derivative estimate in a larger domain $\Om_\beta$.
Moreover, we may assume $\Phi^*$ is symmetric in the sense
$\Phi^*(z)= (\Phi(z) + \overline{\Phi(\bar{z})})/2$ for some~$\Phi$ satisfying the same bounds.
 Let $\Phi^*(z):= u(z) + i\tilde u(z)$.
Then $V^*(x)=0$, $x\in \bR\cap \Om_\al^{1/2}$,  $\tilde u(x)=0$, $x\in \bR\cap \Om_\beta$ , so we have
$$
|U(x) - u(x) |=|U^*(x)-u(x)|= |F^*(x) - \Phi^*(x) |\le \eps,\quad x\in (0, \delta(\eps)].
$$
We now use that
$$
|\nabla \tilde u(z)| \approx |\pd \Phi^*(z)| \le \frac{C_1}{|z|},\quad z\in \Om_\beta,
$$
in conjunction with
$$
\tilde u (x)=0,\, x\in \bR\cap \Om_\beta,
$$
to conclude that
\begin{equation}
\label{bdd}
|\tilde u(z)|=|\tilde u(z)-\tilde u (x)|\le |y|\frac{C_2}{|z|}  \le  C_2, \quad z=x+iy\in \Om_{\beta}.
\end{equation}

We have deduced Theorem \ref{realpart} from  Lemma \ref{fixed}.
\medskip

Next, we present two lemmas that  will be used in our proof of
Lemma \ref{fixed}.  While these two lemmas are certainly well known to experts, we include the proofs
since we do not know good references.

\begin{lemma}
\label{JB}
Let $\vf$ on $I_0:=[-1, 1]$ have Lipschitz norm $N$.  Then $\vf$ can be approximated by polynomials of degree $N^{3/2}$ with the error at
most $cN^{-1/2}$.
\end{lemma}

\begin{proof}
Given $\theta \in [-\pi, \pi]$ we introduce the new function $\Phi(e^{i\theta})= \vf (\cos \theta)$.
Since $|\cos\theta_1 -\cos\theta_2|\le | e^{i\theta_1}- e^{i\theta_2}|$,
 the modulus of continuity of
$\Phi$ is not greater than the modulus of continuity $\omega_\vf$ of $\vf$.
Therefore, using the Jackson--Bernstein theorem we can find a trigonometric polynomial $S(e^{i\theta})$
of degree $K$ such that
$$
|\Phi(e^{i\theta})- S(e^{i\theta})|\le A \om_\vf(1/K).
$$
Notice that $\Phi(e^{i\theta})$ is even by construction, so $S$ can be just a linear combination of $\cos k\theta$, $k=0,\dots, K$. Now we substitute $\theta= \arccos x$, and get the combination  $P_K$ of Chebyshev polynomials $T_k(x)= \cos (k\arccos x)$, $k= 0, \dots, K$, such that
$$
|\vf(x)- P_K(x) |\le A \om_\vf(1/K).
$$
Applying this inequality to a Lipschitz function $\vf$ with Lipschitz constant at most $N$ and with $K= N^{3/2}$
completes the proof.
\end{proof}

\begin{lemma}
\label{pol}
Let a polynomial $P$ of degree $d$ satisfy $|P(x)|\le 1$, for $x\in I_0:=[-1, 1]$.
Then \textup{1)}~$P$ satisfies the uniform estimate~$|P(z)| \leq |16z|^d$ when~$|z| \geq 1$\textup{;} \textup{2)} $P'$ satisfies the same estimate with a slightly bigger constant in place of~$16$.
\end{lemma}

\begin{proof}
We use the Lagrange interpolation formula
\begin{equation}\label{Lagrange}
P(z) = \sum\limits_{j=0}^d P\Big(\frac{j}{d}\Big)\frac{\prod_{i\ne j}(z-\frac{i}{d})}{\prod_{i\ne j}(\frac{j-i}{d})}
\end{equation}
to establish the inequality (for~$|z| > 1$)
\begin{equation*}
|P(z)| \leq |2z|^d\Big(\sum\limits_{j=0}^d \frac{d^d}{j!(d-j)!}\Big) = |2x|^d \frac{(2d)^d}{d!} \leq |16 z|^d
\end{equation*}
since~$d^d < 4^d d!$. The estimate for the derivative can be proved in the same manner after one differentiates~\eqref{Lagrange}.
%Function $\frac1d\log |P(z)|$ is subharmonic, and it is dominated by  Green's function $G(z)$ on $I_0$,
%where $G$ is the Green's function (with pole at infinity) of $\bC\setminus I_0$. At infinity it is also dominated by $G$. Thus by maximal principle we get
%$$
%|P(z)| \le e^{d G(z)}\,.
%$$
%It is left to write a simple formula for $G(z)$ and to see that the lemma is proved. One can also do this without the formula, as it is clear that $G(z) \le a_0 + \log|z|$ with an absolute $a_0$.

\end{proof}

\subsection{The proof of Lemma \ref{fixed}}
\label{prooffixed}

Let us change the variable: use $w$ for the variable in $\Om_\beta$ and put $w=e^{-z}$, where
$$
z\in \Pi_\beta:=\{ z=x+iy: x>0, |y|<\beta/2\}.
$$
The same change of variable relates $\Pi_\al^{\log 2} $ to $\Om_\al^{1/2}$:
$$
z\in \Pi_\al^{\log 2} :=\{ z=x+iy: x> \log 2, |y|<\al/2\}.
$$
  Condition \eqref{Bl} for $F$ becomes the following condition for $f(z):= F(e^{-z})$ in $\Pi_\al^{\log 2}$:
\begin{equation}
\label{BlP}
|\pd f(z)| \le C_1,\, z\in \Pi_\al^{\log 2},
\end{equation}
which is precisely the Lipschitz assumption on a function $f$ analytic in $\Pi_\al^{\log 2}$.

We need to find  {\it analytic} $h\in \Pi_\beta$ such that for some large  number $\Delta(\eps)$
\begin{equation}
\label{h}
|f(x)-h(x)| \le \eps, \,\,x\ge \Delta(\eps)\quad\text{and}\quad |\pd h(z)| \le C_2,\, z\in \Pi_\beta.
\end{equation}

We begin with a Lipschitz  extension of $f$ from $\Pi_\al^{\log 2}$ into $\Pi_\al$.
For example we can extend $f$ by symmetry with respect to the vertical line $x=\log 2$.
Namely,  $f$ extends by \eqref{BlP} to be continuous on the closure of ${\Pi_\al^{\log 2}}$ and then,
given $z=x+iy, 0<x<\log 2,$ we define $z^*= (2\log 2 -x)+iy$ and put
$$
f^*(z)=\begin{cases} f(z^*), \, z=x+iy, 0<x<\log 2, \, z\in \Pi_\al;\\
f(z), \, \,z=x+iy, x\ge \log 2, \,z\in \overline{\Pi_\al}\,.
\end{cases}
$$
It is easy to see that the new function $f^*$ is a Lipschitz function in the whole strip $\Pi_\al$, and it extends the analytic function $f$ defined on $\Pi_\al^{\log 2}$.
Then $f^*$ is not differentiable at the points
$z=\log 2+iy\in \Pi_\al$, but a standard smoothing of $f^*$ will have bounded gradient on $\Pi_\al$
and be an extension of the restriction of $f$ to $\Pi_\al^{1} :=\{ z\in \Pi_\al: \re z> 1\}$.
Below the symbol $f$ denotes this extension.

Consider now
$$
H(x, y):= f\big(x+ i \frac{\al}{\beta} y\big), \, x+iy \in \Pi_\beta.
$$

It satisfies
$$
|\nabla H(z)|\le C_3,\, \, z \in \Pi_\beta \quad\text{and}\quad  |f(x)- H(x)|=0,\,\, x\ge  0,
$$
but it is not analytic.

We claim that there is a function $g$ such that
\begin{equation}
\label{gH}
\bar{\pd} g =\bar{\pd} H;\text{ and } |g(x)| \le \eps, x\ge \Delta(\eps);
\text{ and } |\pd g(z)| \le C_4,\, z\in \Pi_\beta.
\end{equation}

The existence of such a function $g$ will complete the proof of Lemma \ref{fixed}.  Indeed, setting
$$
h:= H-g,
$$
we have for $x\ge \Delta(\eps)$ that
$$
|f(x)- h(x)|= |H(x)- (H(x)- g(x))|=|g(x)|\le \eps.
$$
Also $h$ is analytic: $ \bar{\pd} h = \bar{\pd} H -\bar{\pd} g =0$. Moreover
$$
|\pd h |= |\pd  H- \pd g| \le C_3+ C_4\,.
$$
This will establish that  (\ref{h}) holds, and thus complete the proof of Lemma \ref{fixed}.
Hence it suffices to construct a function $g$ that satisfies (\ref{gH}).

\subsection{Analytic partition of unity}

There exists a number $b>0$ such that in $\Pi_\beta$ the function
$$
w(z)=\sum_{k=0}^\infty e^{-(z/b -k)^2}
$$
is uniformly bounded  away from zero in absolute value. In fact, $\sum_{k=0}^\infty e^{-(z -k)^2}$ is
bounded away from zero on $\bR_+$ and its derivative obviously is uniformly bounded in a fixed thin strip around $\bR_+$.
Then in a smaller but fixed strip it is uniformly bounded away from zero in absolute value. Thus, with
$b$ chosen to be sufficiently large, $w$  will  be uniformly bounded
away from zero in absolute value on $\Pi_\beta$.
We now introduce the notation
$$
 e_k(z)=  e^{-(z/b -k)^2}/ w(z)\, ,
$$
and note that
\begin{equation}
\label{modul}
\sum_{k=0}^\infty e_k(z) =1 \quad\text{and}\quad\sum_{k=0}^\infty |e_k(z)| \le C_5,\quad z\in\Pi_\beta\,.
\end{equation}

\subsection{The first modification of $H$}

As a step toward \eqref{gH}, let us first modify $H$ to $H_0$ in $\Pi_\beta$ in such a way that $\bar{\pd} H_0 =\bar{\pd} H$, but that also
\begin{equation}
\label{bddH}
|\nabla H_0(z)|\le C_6,\,\,|H_0(z) |\le C_7\,\quad z\in\Pi_\beta\,.
\end{equation}
Here is the formula for $H_0(z)$:
$$
H_0(z) = \sum_{k=0}^\infty e_k(z) (H(z)- H(k))\,.
$$
The Lipschitz property of $H$ (remember that $|\nabla H|\le C_3 $ in $\Pi_\beta$) and
\eqref{modul} prove that $H_0$ is bounded, and of course $\bar{\pd} H=\bar{\pd} H_0$.
We also have that
\begin{equation}
\label{modul2}
\sum_{k=0}^\infty |\pd e_k(z)| \le C_7,\quad z\in\Pi_\beta\,,
\end{equation}
which can be used  to estimate $|\nabla H_0|$ in the same way \eqref{modul} was used to estimate $| H_0|$.

\subsection{The second step of the modification of $H$, from $H_0$ to $g$}

Re-writing (\ref{gH}) in terms of $H_0$, we need to find $g$ such that
\begin{equation}
\label{gH0}
\bar{\pd} g =\bar{\pd} H_0; \text{ and } |g(x)| \le \eps, x\ge \Delta(\epsilon); \text{ and } |\pd g(z)| \le C_4,\, z\in \Pi_\beta.
\end{equation}

Let $m$ be a large integer to be fixed later.  Consider functions
$$
h_{k, m} (t):= H_0(btm+ bk), \,\, t\in[-1,1],
$$
where $k$ and $m$ are integers and $b$ is the parameter from the partition of unity introduced in
section 3.3.
These are functions on the interval $I_0:= [-1,1]$ with Lipschitz norm bounded by $Cbm$, where $C=C_6$ from \eqref{bddH}.

We now apply Lemma \ref{JB} to the  functions $h_{k, m}$ defined above to get polynomials $\cP_k$ of degree
bounded by  $ (Cbm)^{3/2}:=\lambda m^{3/2}$ such that
\begin{equation}
\label{a1}
|\cP_k(t) -  h_{k, m}(t) |\lesssim m^{-1/2}, \quad t\in I_0,
\end{equation}
which translates to
\begin{equation}
\label{a2}
 \Big|\cP_k\big(\frac{x-bk}{bm}\big) -  H_0(x)\Big|\lesssim m^{-1/2},
\end{equation}
whenever $|x-bk|\le bm$.

We can now give the formula for the function $g$ that will satisfy (\ref{gH0}):
$$
g(z):= \sum_{j=0}^\infty e_j(z) \Big(H_0(z) - \cP_j\big(\frac{z-bj}{bm}\big)\Big).
$$
From (\ref{modul})  it is clear that $\bar{\pd} g =\bar{\pd} H_0$, so it remains to estimate
 1)  $|g(x)|$ when $x\in \bR$ is large, and 2) $|\pd g(z)|$ when $z\in \Pi_\beta$.

\medskip

Fix $x_0>0$ and let $k_0$ be the integer such that $|x_0- bk_0|\le b$.
We split the sum in the definition of $g(x_0)$ into three parts:

$$
\Sigma_1:= \sum_{j: |j-k_0| \le m-10} e_j(x_0) \Big(H_0(x_0) - \cP_j\big(\frac{x_0-bj}{bm}\big)\Big);
$$
$$
\Sigma_2:= \sum_{j: j-k_0 \ge m-9} e_j(x_0) \Big(H_0(x_0) - \cP_j\big(\frac{x_0-bj}{bm}\big)\Big);
$$
$$
\Sigma_3:= \sum_{j\ge 0: k_0-j \ge m-9} e_j(x_0) \Big(H_0(x_0) - \cP_j\big(\frac{x_0-bj}{bm}\big)\Big).
$$
For the indices $j$ occurring in $\Sigma_1$, we have
$$
\frac{|x_0 - bj|}{bm} \le \frac{|x_0-bk_0 |}{bm}+ \frac{|bk_0 - bj|}{bm}
\le\frac1m+ \left(1- \frac{10}{m}\right)  \le 1.
$$
Hence \eqref{a2} applies to each term in $\Sigma_1$. As the sum $\sum_{j\ge 0} e_j(z)$ converges absolutely in our strip, we get that
\begin{equation}
\label{s1}
\left| \Sigma_1 \right|\le Cm^{-1/2}.
\end{equation}

To estimate $\Sigma_2$ and $ \Sigma_3$ we need the following estimate of $\cP_r$ and $\cP'_r$, $r\ge 0$:
\begin{equation}
\label{Pr}
|\cP_r(z)|+|\cP'_r(z)|\le (C|z|)^{\lambda m^{3/2}}, \quad |z|\ge 1.
\end{equation}
Here the constant  $C$ is independent of $r\ge 0$.
This follows from Lemma \ref{pol} and \eqref{a1}, since $H_0$ is bounded.

Next, notice that the part of $\Sigma_2$, $\Sigma_2':= \sum_{j: j-k_0 \ge m-9} e_j(x_0) H_0(x_0)$ is obviously small if $m$ is large. In fact, $|H_0|\le C_7$ from (\ref{bddH}), and so
\begin{align*}
|\Sigma_2'|&\le C_7\sum_{j: j-k_0 \ge m-9} |e_j(x_0)|
\lesssim C_7 \sum_{j: j-k_0 \ge m-9} e^{-\frac{|x_0- bj|^2}{b^2}}\\
&\le C\sum_{j: j-k_0 \ge m-9} e^{-\frac{|bk_0- bj|^2}{b^2}}= C\sum_{j: j-k_0 \ge m-9} e^{-|k_0- j|^2}\,
\end{align*}
which converges to zero as $m\to\infty$ by convergence of the series.
To estimate the other part of $\Sigma_2$, namely $\Sigma_2'':= \sum_{j: j-k_0 \ge m-9} e_j(x_0) \cP_j(\frac{x_0-bj}{bm})$,
we notice that for the indices involved, we have (as $|x_0-bk_0|\le b$)
$$
\frac{|x_0-bj|}{bm}\le m^{-1}(1+ |j-k_0|)\le 1+ |j-k_0|.
$$
Therefore,
from \eqref{Pr} we get
\begin{equation}
\label{cPj}
\Big|\cP_j\big(\frac{x_0-bj}{bm}\big)\Big| \le (C+C|j-k_0|)^{\lambda m^{3/2}}.
\end{equation}
For later use, we note that the derivative can be estimated in the same way, using \eqref{Pr}:
\begin{equation}
\label{diffcPj}
\Big|\cP'_j\big(\frac{z-bj}{bm}\big)\Big| \le C(\beta)b m (C+C|j-k_0|)^{\lambda m^{3/2}},
\quad |z-x_0| \le \beta.
\end{equation}

Hence,
\begin{align*}
|\Sigma_2''|&\le \sum_{j: j-k_0 \ge m-9} |e_j(x_0)| |\cP_j(\frac{x_0-bj}{bm})| \le C \sum_{j: j-k_0 \ge m-9} e^{-\frac{|x_0- bj|^2}{b^2}}|\cP_j(\frac{x_0-bj}{bm})| \\
& \lesssim
 \sum_{j: j-k_0 \ge m-9}\!\!\!\!\!\!(C+ C|j-k_0|)^{\lambda m^{3/2}} e^{-\frac{|bk_0- bj|^2}{b^2}}\\
 &= \sum_{j: j-k_0 \ge m-9}\!\!\!\!\!\! (C+ C|j-k_0|)^{\lambda m^{3/2}}  e^{-|k_0- j|^2}.
\end{align*}

This is small if $m$ is chosen to be large, and combined with the previous estimate for $|\Sigma_2'|$
 we get that $|\Sigma_2|\to0$ as $m\to\infty$.
 Notice that the same argument shows that $\Sigma_3$ is small when $m$ is large. Combining these estimates with \eqref{s1}, we obtain
\begin{equation}
\label{small_g}
 |g(x)| =|\Sigma_1+\Sigma_2+\Sigma_3|\le\eps, \quad\text{whenever}\quad x\ge bm\,,
\end{equation}
where $\eps=\eps(m)\to 0$ as $m\to \infty$.  This is the required estimate for $|g(x)|$.

%%%%%%%%%%%%%%%%%%%%%%%%%%%%

\subsection{The estimate of  $\pd g$}
\label{deriv}
It remains to estimate $\pd g$.
The terms in $\pd g$ with $\frac{w'(z)}{w^2(z)}$ are estimated precisely as before,
since $\frac{w'(z)}{w^2(z)}$ is bounded on $\Pi_\beta$.
The terms with $(z/b-k) e^{-(z/b-k)^2}$ can be estimated along verbatim the same lines as before.

What is left, is to estimate
$$
\sum_{k=0}^\infty e_k(z) (\pd H_0(z) - \frac{1}{bm}\cP'_k(\frac{z-bk}{bm})).
$$
From the estimate $|\nabla H_0|\le C_6$ and the absolute convergence of $\sum e_k$, we need only to prove
$$
\sum_{k=0}^\infty |e_k(z)| |\cP'_k(\frac{z-bk}{bm})|\le C,\quad z\in \Pi_\beta.
$$

Let $z= x_0 +iy, |y|\le \beta/2$, and $k_0$ an integer with $|x_0-bk_0|\le b$.
Using \eqref{diffcPj}, we can estimate
$$
\sum_{k=0}^{\infty} |e_k(z)| |\cP'_k(\frac{z-bk}{bm})|
\lesssim
C(\beta) bm\sum_{k=0}^{\infty} |e_k(bk_0)| (C+ C|k-k_0|)^{\lambda m^{3/2}}.
$$
Since $|e_k(bk_0)|\le C \exp(-|k-k_0|^2)$,
this sum is clearly bounded by some $C(m, \beta, b)<\infty$.

This completes the proof of the estimate $|\pd g(z)|\le C(m,\beta,b)$, $z\in \Pi_\beta$.
Hence (\ref{gH0}) has been established,  the proof of the lemma is complete.

\begin{remark} In fact, one can modify the $\bar\pd $ proofs of this section to get a better claim. Namely, one can obtain the following statement by modifying the proofs above.
Let $0<\alpha'  <\alpha<\beta<\pi/2$, $\eps>0$, and $0<\delta<1/2$. Given a function
$F \in \mathcal B(\Omega_\alpha^{1/2})$ one can find analytic $\Phi$ such that
1) $|F(z)-\Phi(z)|\le\varepsilon$, $z\in \Omega_{\alpha'} \cap \{w\in \bC : |w|\le\delta\}$; 2) $\Phi \in \mathcal B(\Omega_\beta)$ and $C_\Phi = C(\varepsilon,\delta,\alpha',\alpha,\beta,C_F)$.
\end{remark}

We are grateful to the referee for telling us that this better result was available.

\section{Examples: the proof that BRV $\subsetneqq T[H^\infty]$}
\label{examples}

We now present the example showing that BRV $\subsetneqq T[H^\infty]$.  Again, we thank
F. Nazarov,  P. Ivanisvili, and A. Logunov for permission to include it here.

\begin{prop}\label{MainEx}
There exists an analytic function~$g\colon \mathbb{D} \to \mathbb{C}$ such that the operator~$T_g$ is bounded on~$H^{\infty}(\mathbb{D})$ and
\begin{equation*}
\int\limits_{-1}^0 |g'(z)|\,dz = +\infty,
\end{equation*}
where integration is along the radius~$(-1,0]$.
\end{prop}

It will be more convenient to work with another domain. Let $\mathbb{D}_1$ be the disk with center $1$ and
radius 1, and  denote by $\log$ the branch of the logarithm on
$\mathbb{D}_1$ that preserves the real numbers. Then
\begin{equation*}
z\in\mathbb{D}_1 \quad \Leftrightarrow \quad \zeta \in \Omega = \Big\{(x,y) \in\mathbb{R}^2\,:\; x < \log(2\cos y), y \in \Big(\frac{\pi}{2},\frac{\pi}{2}\Big)\Big\}
\end{equation*}
if~$\zeta = \log z$.

\begin{prop}\label{StOmega}
There exists an analytic function~$f\in H^{\infty}(\Omega)$ such that
\begin{equation}\label{CondOnf}
\int\limits_{-\infty}^{0}|f(\xi)|\,d\xi = \infty, \quad \hbox{but}\quad \bigg|\int\limits_{\zeta}^{0}f(\xi) h(\xi)\,d\xi\bigg| \lesssim \|h\|_{H^{\infty}(\Omega)}
\end{equation}
for  all $\zeta $ real and negative, and  for all $h \in H^{\infty}(\Omega)$.
\end{prop}

\begin{proof}[Proof of Proposition~\textup{\ref{MainEx}} assuming Proposition~\textup{\ref{StOmega}}:]
We claim that the function
$$
g(z) = \int_1^z \frac{f(\log w)}{w}\,dw, \quad w\in \mathbb{D}_1,
$$
is almost the one described in Proposition~\ref{MainEx}
(the only difference is that its domain is~$\mathbb{D}_1$). First, its variation along the radial segment~$(0,1]$ is infinite:
\begin{equation*}
\int\limits_{0}^1 |g'(r)|dr = \int\limits_{0}^1 \Big|\frac{f(\log r)}{r}\Big|dr = \int_{-\infty}^0|f(\xi)|\,d\xi = +\infty.
\end{equation*}
Second, for any function~$\tilde{h} \in H^{\infty}(\mathbb{D}_1)$,
\begin{equation*}
-T_g[\tilde{h}](z) = \int\limits_{z}^1 \tilde{h}(s)g'(s)\,ds = \int\limits_{\zeta}^0 \tilde{h}(e^{\xi})f(\xi)\,d\xi,\quad \zeta = \log z.
\end{equation*}
We did not specify the curve of integration in the line above because the integral is path independent.
Since $|\zeta|<\log2 + \pi/2$ when $\re \zeta \geq 0$,
it is clear that the  integral above is bounded by~$C\|f\|_{H^{\infty}}\|\tilde{h}\|_{H^{\infty}}$
in this case.

For the case that $\re \zeta<0$, we pick a specific contour: starting at~$\zeta \in \Omega$, we
integrate first along the vertical segment~$v_{\zeta} = [\zeta,\re \zeta]$, and then integrate along
the horizontal segment~$[\re \zeta, 0]$. This leads to the bound (let~$\tilde{h}(e^{\xi})$ be simply~$h(\xi)$)
\begin{align*}
\big|T_g[\tilde{h}](z)\big| &\leq \bigg|\int\limits_{v_{\zeta}} h(\xi)f(\xi)\,d\xi\bigg| + \bigg|\int\limits_{\re \zeta}^0 h(\xi)f(\xi)\,d\xi\bigg| \\
&\lesssim \|f\|_{H^{\infty}}\|h\|_{H^{\infty}} + \|h\|_{H^{\infty}} \lesssim \|\tilde{h}\|_{H^{\infty}}.
\end{align*}
To bound the first summand we used that the
length of $v_{\zeta}$ is at most ${\pi}/{2}$; the bound for the second summand came from
Proposition~\textup{\ref{StOmega}}.
Now we only have to shift~$\mathbb{D}_1$ to transfer~$g$ to~$\mathbb{D}$.
\end{proof}

\begin{proof}[Proof of Proposition~\textup{\ref{StOmega}}]
The function~$f$ will be given by the formula
\begin{equation}\label{FormulaForf}
f(\xi) = \sum\limits_{k=0}^{\infty} a_k e^{i\lambda_k(\xi - \zeta_k)}e^{-(\xi - \zeta_k)^2},
\end{equation}
where the sequence~$\{\lambda_k\}_k$ is real-valued and tends to~$+\infty$, and the sequence~$\{\zeta_k\}_k$ is real-valued and tends rapidly to~$-\infty$. We require that the following conditions are satisfied:
\begin{equation}\label{SumAk}
\sum\limits_{k=0}^{\infty}|a_k| = +\infty;
\end{equation}
\begin{equation}\label{FiniteSum}
\sum\limits_{k=0}^{\infty}\frac{|a_k|}{\lambda_k} < +\infty;
\end{equation}
\begin{equation}\label{Product}
|a_k|e^{2\lambda_k} \lesssim\ 1.
\end{equation}
For example, we may take~$a_k = e^{-2\lambda_k}$,~$\lambda_k = \frac12(\log k + \log\log k)$. With this choice of~$\lambda_k$ and~$a_k$, we may take~$\zeta_k = -2^k$. We will work with this particular choice of the sequences, however, the main properties we will use are~\eqref{SumAk}, \eqref{FiniteSum}, \eqref{Product}, and the fact that~$\{\zeta_k\}_k$ decreases sufficiently fast.

First, the function~$f$ is uniformly bounded on $\Omega$:
\begin{equation*}
|f(\xi)|\leq\Big|\sum\limits_{k=0}^{\infty} a_k e^{i\lambda_k(\xi - \zeta_k)}e^{-(\xi - \zeta_k)^2}\Big|
{\lesssim} \sum\limits_{k=0}^{\infty}|a_k|e^{\frac{\pi}{2}\lambda_k}\left|e^{-(\xi -\zeta_k)^2}\right|\lesssim 1,
\end{equation*}
where \eqref{Product} was used to get the last estimate.

Next, we prove the first part of~\eqref{CondOnf}.

By~\eqref{SumAk}, it suffices to show that
\begin{equation*}
|f(\xi)| \gtrsim |a_k|,\quad \xi \in [\zeta_k-1,\zeta_k + 1],
\end{equation*}
provided~$k$ is sufficiently large. This is easy: for $\xi \in [\zeta_k-1,\zeta_k + 1]$,
\begin{equation*}
\begin{aligned}
|f(\xi)| &\geq  \frac{|a_k|}{e^2} - \sum\limits_{l < k}|a_l|e^{-(\zeta_k - \zeta_{l} + 1)^2} - \sum\limits_{l > k}|a_l|e^{-(\zeta_l - \zeta_k + 1)^2}\\& \geq
\frac{|a_k|}{e^2} - e^{-2^k}\sum\limits_{l < k}\frac{1}{l\log l} - \sum\limits_{l > k}\frac{1}{l\log l}e^{-(2^l - 2^k+1)^2} \gtrsim \frac{1}{k\log k} = |a_k|.
\end{aligned}
\end{equation*}

Last, we prove the inequality in~\eqref{CondOnf}.
Let $\zeta\in(-\infty,0)$.
Due to~\eqref{FiniteSum}, it suffices to prove that
\begin{equation*}
\bigg|\int\limits_{\zeta}^{0}e^{i\lambda_k(\xi - \zeta_k)}e^{-(\xi - \zeta_k)^2} h(\xi)\,d\xi\bigg| \lesssim \frac{\|h\|_{H^{\infty}(\Omega)}}{\lambda_k}.
\end{equation*}
Integration by parts gives
\begin{align*}
\int\limits_{\zeta}^{0}e^{i\lambda_k(\xi - \zeta_k)}e^{-(\xi - \zeta_k)^2} h(\xi)\,d\xi = &\frac{1}{i\lambda_k}e^{i\lambda_k(\xi - \zeta_k)}e^{-(\xi - \zeta_k)^2} h(\xi)\bigg|_{\xi = \zeta}^0 \\
   &- \frac{1}{i\lambda_k}\int\limits_{\zeta}^0 e^{i\lambda_k(\xi - \zeta_k)}\Big(e^{-(\xi - \zeta_k)^2} h(\xi)\Big)'\,d\xi.
\end{align*}
The first term is clearly bounded by $|\lambda_k|^{-1}\|h\|_{H^{\infty}(\Omega)}$, 
and the required bound for the integration term follows
from the estimate
\begin{equation*}
\bigg|\Big(e^{-(\xi - \zeta_k)^2} h(\xi)\Big)'\bigg| \lesssim (1 + |\xi - \zeta_k|)e^{-(\xi - \zeta_k)^2}\|h\|_{H^{\infty}(\Omega)}.
\end{equation*}
This estimate is a consequence of the inequality $|h'(\xi)|\lesssim \|h\|_{H^\infty(\Omega)}$,
$\xi\in (-\infty,0]$,
which follows
 from the Cauchy integral formula for the derivative.
\end{proof}

\medskip

We end the paper with an example related to the proof of Theorem \ref{oper}.  That proof
would have been easier if  diam$_I\, \vf(\bD)=\infty$ implied that there is a radius $[0,e^{i\theta})$ of $\bD$
such that $\vf( [0,e^{i\theta}))$ is not rectifiable.
Here is an example that shows this may not be the case.

\begin{ex}\label{examp}
Form the domain $$O= U\setminus\bigcup_{k=2}^\infty \ell_k,$$ from
the domain $U= \{x+iy\,:\, 0<x<\infty,\,0<y<e^{-x}\}$ with the rays
$\ell_k=\{x+ie^{-k}/2\,:\, 2\le x<\infty\}$ removed, and let $\vf:\bD\to O$ be a
Riemann map with $\vf(0)=z_0=1+ie^{-1}/2$.
Then the image under $\vf$ of every radius of $\bD$
is rectifiable, but diam$_I \,O=\infty$.  Furthermore, $O$ may be modified to obtain a bounded
domain with the same properties.
\end{ex}

\begin{proof}
That diam$_I\, O=\infty$ is clear.  Now consider any of the degenerate prime ends of $O$,
i.e. any prime end that corresponds to a single point $Q$ on the boundary of $O$.  Clearly there is a rectifiable
curve in $\gamma_Q\subset O$ connecting $Q$ to the point $z_0$ and with length $\Lambda_1(\gamma_Q)< 2+|Q|$.
Let $e^{i\theta}$ be the point on the unit circle that corresponds to $Q$ under the map $\vf$.
By a theorem of Gehring and Hayman (see \cite{GH}, or \cite[Theorem 4.20]{Pom}), there is an absolute constant
$K$ such that
$\Lambda_1(\vf[0,e^{i\theta}))\le K   \Lambda_1(\gamma_Q)< K(2+|Q|)<\infty$.  There is only one  prime
end left to consider, the one with impression $[2,\infty)$ and principal point $P=2$.  Then, for the point $e^{i\eta}$
corresponding to this prime end under the map $\vf$, we have
$$
\lim_{r\to 1}\vf(re^{i\eta})=P,
$$
see \cite[Theorem 2.16]{Pom}, and once again the theorem of Gehring and Hayman tells us that
$\Lambda_1(\vf[0,e^{i\eta}))<\infty$.

We now modify $O$ to obtain a bounded domain with the same properties. First form the domain
$$
\widetilde O = \{x+i(y+e^{-x})\,:\, x+iy\in O\},
$$
and let $\widetilde\vf$ be a Riemann map with $\widetilde\vf(0)=z_0=1+i3e^{-1}/2$.  Next, notice that the restriction
of the function $e^{iz}$ to $\widetilde O$ is univalent.    Indeed, if
$e^{iz_1}=e^{iz_2}$, then $z_2= z_1+2n\pi$ for some integer $n$.  We need to show that
at most one of these points can
be in $\widetilde O$.  So assume that $n>0$ and $z_1= x+iy\in \widetilde O$.  Then
$
e^{-x}<y<2e^{-x},
$
and hence
$
y>2e^{-(x+2n\pi)}.
$
This means that $z_2\notin \widetilde O$, which completes that demonstration that $e^{iz}$ is univalent on
$\widetilde O$.

Applying to $\widetilde\vf$ the
 analysis that was just applied to $\vf$ on $O$ now
shows that  the image under $\widetilde\vf$ of every radius of $\bD$
is rectifiable.  This is preserved under composition with the map $e^{iz}$, and hence
the Riemann map $e^{i\widetilde\vf}$ from $\bD$ to the bounded
 domain $\exp(i\widetilde O)$ that spirals out to the unit circle is the example we are looking for.
\end{proof}

\end{document}